\documentclass[runningheads]{llncs} 
\usepackage{graphicx} 
\usepackage{amsmath}
\usepackage{amssymb}
\usepackage{amsfonts}
\usepackage{todonotes}

\title{Minimality, transitivity and sensitivity of non-uniform cellular automata\thanks{This work was partially supported by the Academy of Finland grant 354965
and by the
HORIZON-MSCA-2022-SE-01 project 101131549 
(ACANCOS).}}
\author{Supreeti Kamilya\inst{1} \and Jarkko Kari\inst{2} \and Katariina Paturi\inst{2}}
\authorrunning{S. Kamilya, J. Kari and K. Paturi}
\institute{Department of Computer Science and Engineering, Birla Institute of Technology, Mesra, Ranchi, India \and Department of Mathematics and Statistics, University of Turku, Finland}

\newcommand{\N}{\mathbb{N}}
\newcommand{\Z}{\mathbb{Z}}
\newcommand{\R}{\mathcal{R}}
\newcommand{\cyl}{\mathrm{Cyl}}

\begin{document}

\maketitle

\begin{abstract}
Every transitive cellular automaton (CA) is sensitive to initial conditions. We study this implication in the more general context of non-uniform cellular automata (NUCA) with finitely many different local update rules assigned to cells. We construct a two-dimensional NUCA that is minimal -- and hence transitive -- but that is not sensitive to initial conditions. The construction is based on an odometer NUCA on $\{0,1,2\}^\N$ which is nearly uniform in the sense that only the first cell uses a different local rule. Then we show that if the assignment of local rules in the cells is recurrent then transitivity implies sensitivity.

\end{abstract}

\section{Introduction}

Classical cellular automata (CA) are homogeneous both in space and time. All cells use the same local update rule, and the same rule is used at all discrete time steps. A non-uniform cellular automaton (NUCA) is a generalization where different cells may use different update rules. However, in our formalism, there are only finitely many different rules available, and the assignment of rules -- called the rule distribution -- does not change over time. Non-uniformity greatly changes possible behaviors of the system. While every CA has temporally periodic points, there are NUCAs -- even reversible equicontinuous ones --  that do not have any temporally periodic points~\cite{KamilyaK21}. Also the classical Garden of Eden theorem~\cite{Moore1962,Myhill1963} of cellular automata that links injectivity and surjectivity does not hold for all NUCAs. However, if the local rule distribution is uniformly recurrent, or recurrent in the one-dimensional case, the Garden of Eden theorem holds~\cite{PaturiNACO}. A one-dimensional CA is either sensitive or almost equicontinuous~\cite{kurka1997}, but this dichotomy breaks for one-dimensional NUCA~\cite{PaturiUCNC}. 

In this work we investigate the relationship of transitivity and sensitivity in NUCA. If a uniform cellular automaton is transitive then it is also sensitive~\cite{CodenottiMargara96}. We show that this implication holds also among NUCA with a recurrent rule distribution. On the other hand, we construct a two-dimensional NUCA (with a non-recurrent rule distribution) that is minimal -- and hence transitive -- but equicontinuous -- and hence not sensitive. In the construction we use a one-way cellular automaton with an odometer behavior. 

\section{Preliminaries}

To define NUCA, we first give some preliminary definitions.

\begin{definition}
Let $\Sigma$ be a finite set with at least two elements, 
called a \emph{state set}, and let $d\in \Z_+$. A \emph{configuration} is an element $c \in \Sigma^{\Z^d}$. A \emph{cell} is an element $\vec{x} \in {Z^d}$.  A \emph{domain} is any subset $D \subseteq \Z^d$.
If $D$ is finite, we call it a \emph{finite domain}. An element $p \in \Sigma^D$ is called a \emph{pattern}, and if $D$ is finite, the pattern $p$ is called a \emph{finite pattern}. The pattern $c_{|D} \in \Sigma^D$ is one such that $c(\vec{x}) = c_{|D}(\vec{x})$ for all $\vec{x} \in D$. 
\end{definition}

\noindent
Next we define shifting of patterns.

\begin{definition}
Let $\vec{x}\in \Z^d$ and let $P= \bigcup_{D \subseteq \Z^d} \Sigma^D$ be the set of patterns with state set $\Sigma$ and dimension $d$. The \emph{($d$-dimensional) $\vec{x}$-shift} is the function $\sigma_{\vec{x}}:P\rightarrow P$ such that for any $D \subseteq \Z^d$ and $p \in \Sigma^D$, the image $\sigma_{\vec{x}} (p) \in \Sigma^{D-\vec{x}}$ is such that for all $\vec{y} \in D-\vec{x}$,
\begin{align*}
    \sigma_{\vec{x}} (p) (\vec{y}) =  p(\vec{y}+\vec{x}).
\end{align*}
\end{definition}

\noindent
In particular, the $\vec{x}$-shift $\sigma_{\vec{x}}(c)$ of a configuration $c \in \Sigma^{\Z^d}$ is a configuration.

\begin{definition}
Let $X \subseteq \Z^d$, $c \in \Sigma^X$ and $D \subseteq X$ be a finite domain. The set
\begin{align*}
    \cyl (c,D) = \{e \in \Sigma^X\ |\ e_{|D}=c_{|D} \}
\end{align*}
is called a \emph{cylinder}.
\end{definition}

For any $X\subseteq\Z^d$, the space $\Sigma^X$  
comes equipped with a metrizable and compact topology. Cylinders form a basis for this topology, that is, any open set is a union of cylinders.

We define NUCA on arbitrary subsets $X$ of $\Z^d$, but generally we are only interested in NUCA on infinite subsets of $\Z^d$, and in particular on $\Z^d$ itself. The odometer NUCA we investigate in Section~\ref{sec:odometer} is defined on $X=\N\subseteq \Z$.

\begin{definition}
    Let $\Sigma$ be a state set and $d \in \Z_+$. Let $N = \{\vec{n}_1, \ldots, \vec{n}_m\}\subseteq \Z^d$ be a finite domain, where $m \in \Z_+$. A \emph{local rule} with neighborhood $N$ is a function $f:\Sigma^N \rightarrow \Sigma$. A finite set of local rules $\R$ is called a \emph{rule set}. 

    Let $X \subseteq \Z^d$ and $\theta \in \R^X$. For any $\vec{x} \in X$, we denote  $N_\theta(\vec{x}) = N_{\theta(\vec{x})} + \vec{x}$, where $N_{\theta(\vec{x})}$ is the neighborhood of the local rule $\theta(\vec{x})$. If for all $\vec{x} \in X$ it holds that $N_\theta (\vec{x}) \subseteq X$, we call $\theta$ a \emph{subset rule distribution}. If $X = \Z^d$, we call $\theta$ a \emph{rule distribution}. 
\end{definition}

\noindent
The requirement $N_\theta (\vec{x}) \subseteq X$ simply means that the neighbors of all cells in $X$ are also in $X$.

\begin{definition}
    A \emph{subset non-uniform cellular automaton} (or \emph{subset NUCA}) on $\Z^d$ is a tuple $A=(\Sigma,d,X,\R,\theta)$, where $\Sigma$ is a state set, $d\in \Z_+$ is a dimension, $X\subseteq \Z^d$ is a domain, $\R$ is a rule set where any rule $f \in \R$ has neighborhood $N_f \subseteq \Z^d$, and $\theta \in \R^X$ is a subset rule distribution. The \emph{global update rule} of $A$ is the function $H_\theta:\Sigma^X \rightarrow \Sigma^X$ such that
    \begin{align*}
        H_\theta (c)(\vec{x}) = \theta (\vec{x}) (\sigma_{\vec{x}}(c_{|N_\theta (\vec{x})})).
    \end{align*}
    If $X=\Z^d$, $A$ is a \emph{non-uniform cellular automaton} (or \emph{NUCA}), and if additionally $|\R| = 1$, the NUCA is a \emph{uniform cellular automaton} (or \emph{CA}).
\end{definition}

\noindent
We generally equate a NUCA to its global update rule. We remark that a subset non-uniform cellular automaton $H_\theta:\Sigma^X \rightarrow \Sigma^X$ is a continuous transformation of the compact metric space $\Sigma^X$, and thus the pair $(\Sigma^X,H_\theta)$ is a discrete time topological dynamical system.

Next, we define (spatially) recurrent rule distributions on $\Z^d$.

\begin{definition}
    Let $d\in \Z_+$ and let $\theta \in \R^{\Z^d}$ be a rule distribution. If for every finite domain $D \subseteq \Z^d$ there exists $\vec{x}\neq\vec{0}$ such that $\sigma_{\vec{x}}(\theta)_{|D}=\theta_{|D}$, then $\theta$ is \emph{recurrent}. 
\end{definition}

Finally, we give definitions for a few dynamical properties. These are direct translations of
the corresponding standard properties on general discrete time topological dynamical systems to our setting. We start with sensitivity and equicontinuity.

\begin{definition}
    Let $\Sigma$ be a state set, $d\in \Z_+$, $X \subseteq \Z^d$, $\R$ a rule set and $\theta \in \R^X$ a subset rule distribution. The subset NUCA $H_\theta$ is 
    \begin{itemize}
        \item \emph{sensitive to initial conditions} (or \emph{sensitive}) if there is some finite domain $E \subseteq X$ such that for any $c \in \Sigma^X$ and any finite domain $D \subseteq X$, there is some $e \in \cyl (c,D)$ such that for some $r \in \N$ we have $H_\theta^r (e) \notin \cyl (H_\theta^r (c),E)$.
        \item equicontinuous if for every finite domain $E \subseteq X$ there is a finite domain $D \subseteq X$ such that for any $c,e \in \Sigma^X$, if $e \in \cyl (c,D)$ then for all $r \in \N$ we have $H_\theta^r (e) \in \cyl (H_\theta^r (c),E)$.
    \end{itemize}
\end{definition}

\noindent
Equicontinuous systems are stable while sensitive ones are unstable. Next we define several variants of transitivity. We give the definitions for general dynamical systems.

\begin{definition}
    Let $S$ be a metric space and $F:S \rightarrow S$ a continuous map.  For any non-empty open sets $U,V \subseteq S$, denote 
    \begin{align*}
        A_{U,V} = \{t \in \N\ |\ F^t(U) \cap V \neq \emptyset \}.
    \end{align*}
    The map $F$ is
    \begin{itemize}
        \item \emph{transitive}, if $A_{U,V} \neq \emptyset$,
        \item \emph{syndetically transitive} if there exists $n\in \Z_+$ such that for any $t \in A_{U,V}$, there is $t' \in A_{U,V}$ with $t < t' \leq t+n$ (in other words, the set $A_{U,V}$ is syndetic),
        \item \emph{strongly transitive}, if $\bigcup_{t \in \N} F^t (U) = S$,
        \item \emph{totally transitive}, if for all $n \in \N$, $F^n$ is transitive,
        \item \emph{weakly mixing}, if the map $F \times F: S^2 \rightarrow S^2$, such that $(F \times F) (x,y) = (F(x), F(y))$ for all $x,y \in S$, is transitive,
        \item \emph{mixing}, if $\N \setminus A_{U,V}$ is finite, 
    \end{itemize}
     for any non-empty open $U,V \subseteq S$.
\end{definition}

\noindent
In the case $F$ is a subset NUCA it is sufficient to consider sets $A_{U,V}$ 
for arbitrary cylinders. 

\begin{definition}
    Let $S$ be a compact metric space and $F:S \rightarrow S$ a continuous map. For any $x \in S$, let $\mathrm{orb}(x) = \{F^t(x)\ |\ t\in \N\}$ be the orbit of $x$. The map $F$ is \emph{minimal} if for all $x\in S$, $\mathrm{orb}(x)$ is dense in $S$.
\end{definition}

For any discrete time dynamical system,  we have the trivial relations that a mixing system is weakly mixing, syndetically transitive and totally transitive, and a system is transitive if it is minimal, weakly mixing, strongly transitive, syndetically transitive, or totally transitive. In uniform CA, transitivity, total transitivity, weakly mixing and syndetic transitivity are mutually equivalent \cite{Moothathu2005}. It is an open problem whether all transitive or strongly transitive uniform CA are mixing. 

The following example NUCA is strongly transitive (and hence transitive) but not totally transitive or weakly mixing (and hence not mixing). It is also syndetically transitive.

\begin{example}
\label{ex:firstexample}
    Let $\Sigma = \{0,1 \}$ and $N = \{-1,0,1\} \subseteq \Z$. Let $\R=\{f_\rightarrow, f_\leftarrow, g\}$ be a set of local rules $ \Sigma^N \rightarrow \Sigma$  where
    $$
    \begin{array}{rcll}
        f_\rightarrow(p) &=& p(-1) &\mbox{(the right shift)}\\
        f_\leftarrow(p) &=& p(1) &\mbox{(the left shift)}\\
        g(p) &=& p(0) \oplus 1 \hspace*{5mm}&\mbox{(the toggle)} 
    \end{array}
    $$
    where $p \in \Sigma^N$ and $\oplus$ is the modulo 2 sum. Let $\theta \in \R^\Z$ be a rule distribution where
    \begin{align*}
        \theta(x) = \begin{cases}
            f_\rightarrow, &\text{if } x < 0,\\
            g, &\text{if } x = 0,\\
            f_\leftarrow, &\text{if } x > 0,
        \end{cases}
    \end{align*}
    where $x \in \Z$. The NUCA $H_\theta$ is strongly transitive (and hence transitive), but neither totally transitive nor weakly mixing (and hence not mixing).
    
    First to show that $H_\theta$ is strongly transitive. Let $c,e \in \Sigma^\Z$ be any two configurations and $D \subseteq \Z$ a finite domain. To prove strong transitivity we show that there is $e' \in \cyl (c,D)$ and $t \in \N$ such that $H_\theta^t (e') = e$. Let $r \in \Z_+$ be even if $c(0)=e(0)$ and odd if $c(0)\neq e(0)$, and such that $D \subseteq [-r,r]$. Let $e' \in \cyl(c,[-r,r]) \subseteq \cyl (c,D)$ be the configuration such that  
    $$
    e'(x) = 
    \left\{
    \begin{array}{ll}
    e(x+r), & \mbox{ for $x < -r$,}\\
    c(x), & \mbox{ for $x\in [-r,r]$,}\\
    e(x-r), & \mbox{ for $x > r$.}
    \end{array}
    \right.
    $$
    For all $x<0$ we have $H_\theta^r (e')(x) = e'(x-r) =  e(x)$. Similarly, for all
    $x>0$ we have $H_\theta^r (e')(x) = e'(x+r) =  e(x)$. And because $H_\theta$ toggles $r$ times the bit at cell $0$, we also have $H_\theta^r (e')(0) = e(0)$. So  $H_\theta^r (e')=e$ as required.

    The NUCA $H_\theta$ is not weakly mixing, because for any configurations $c,e \in \Sigma^\Z$ such that $c(0) = e(0)$, and any $t\in \N$, $H_\theta^t (c)(0) = H_\theta^t (e)(0)$. It is not totally transitive, because $H_\theta^2$ is not transitive. This is because for any $c\in \Sigma$, $H_\theta^2 (c)(0) = c(0)\oplus 1 \oplus 1 = c(0)$. If we have configurations $c_0,c_1 \in \Sigma^\Z$ such that $c_0(0) \neq c_1(0)$, then for all $t \in \N$, $(H_\theta^2)^t(c_0)(0)  \neq c_1(0)$ and hence $(H_\theta^2)^t(\cyl (c_0, \{0\})) \cap \cyl(c_1, \{ 0\})= \emptyset$. \qed
    
\end{example}


\section{A three state odometer on $\N$}
\label{sec:odometer}

Let $\Sigma=\{0,1,2\}$ and $X=\N\subseteq\Z$. Let $f$ be the local rule that uses the neighborhood $\{-1,0\}$ and the transition function given in Table~\ref{table:odometer1_rule1}, and let $g$ be the
local rule with the neighborhood $\{0\}$ and the transition function given in Table~\ref{table:odometer1_rule2}. Let $\theta\in \{f,g\}^\N$ be the local rule distribution $g\ f\ f\ f\dots$ on $\N$, \emph{i.e.}, cell $0$ uses the local rule
$g$ that makes the states cycle $0\mapsto 1\mapsto 2\mapsto 0\mapsto \dots$, while all other cells use the local rule $f$ that (i) keeps the state unchanged if the left neighbor is in state $0$, (ii) swaps states $0\leftrightarrow 1$ if the left neighbor is in state $1$, and (iii)  swaps states $0\leftrightarrow 2$ if the left neighbor is in state $2$. Figure~\ref{fig:odometerspacetime} shows the beginning of the space-time diagram starting from the $0$-uniform initial configuration. We call the global transition function $H=H_\theta:\{0,1,2\}^\N\longrightarrow \{0,1,2\}^\N$ the \emph{three state odometer}.

\begin{table}[h!]
\centering
\begin{tabular}{c|c|c|c}
\textbf{Cell $(i-1)$} & \textbf{Cell $i$} & \textbf{Next state of cell $i$} & 
 \\ \hline
0 & $s$ & $s$ & 
\fbox{\begin{tabular}{c}
\emph{identity}
\end{tabular}}
 \\ \hline

1 & 0 & 1 & 
\\

1 & 1 & 0 & 
\fbox{\begin{tabular}{c|c}
1 & $0 \leftrightarrow 1$
\end{tabular}} 
\\

1 & 2 & 2 & 
\\ \hline
2 & 0 & 2 &\\
2 & 1 & 1 & \fbox{\begin{tabular}{c|c}
2 & $0 \leftrightarrow 2$
\end{tabular}}\\
2 & 2 & 0 &\\
\hline 
\end{tabular}
\bigskip
\caption{The local rule at cells $i>0$ in the three state odometer.}
\label{table:odometer1_rule1}
\end{table}

\begin{table}[h!]
\centering
\begin{tabular}{c|c}
 \textbf{Cell $i$} & \textbf{Next state of cell $i$} 
 \\ \hline
0 & 1 
 \\
1 & 2  
 \\
2 & 0 
\\
\hline 
\end{tabular}
\bigskip
\caption{The local rule at cell $0$ in the three state odometer.}
\label{table:odometer1_rule2}
\end{table}

\begin{figure}
\begin{center}
	\includegraphics[scale=0.3]{"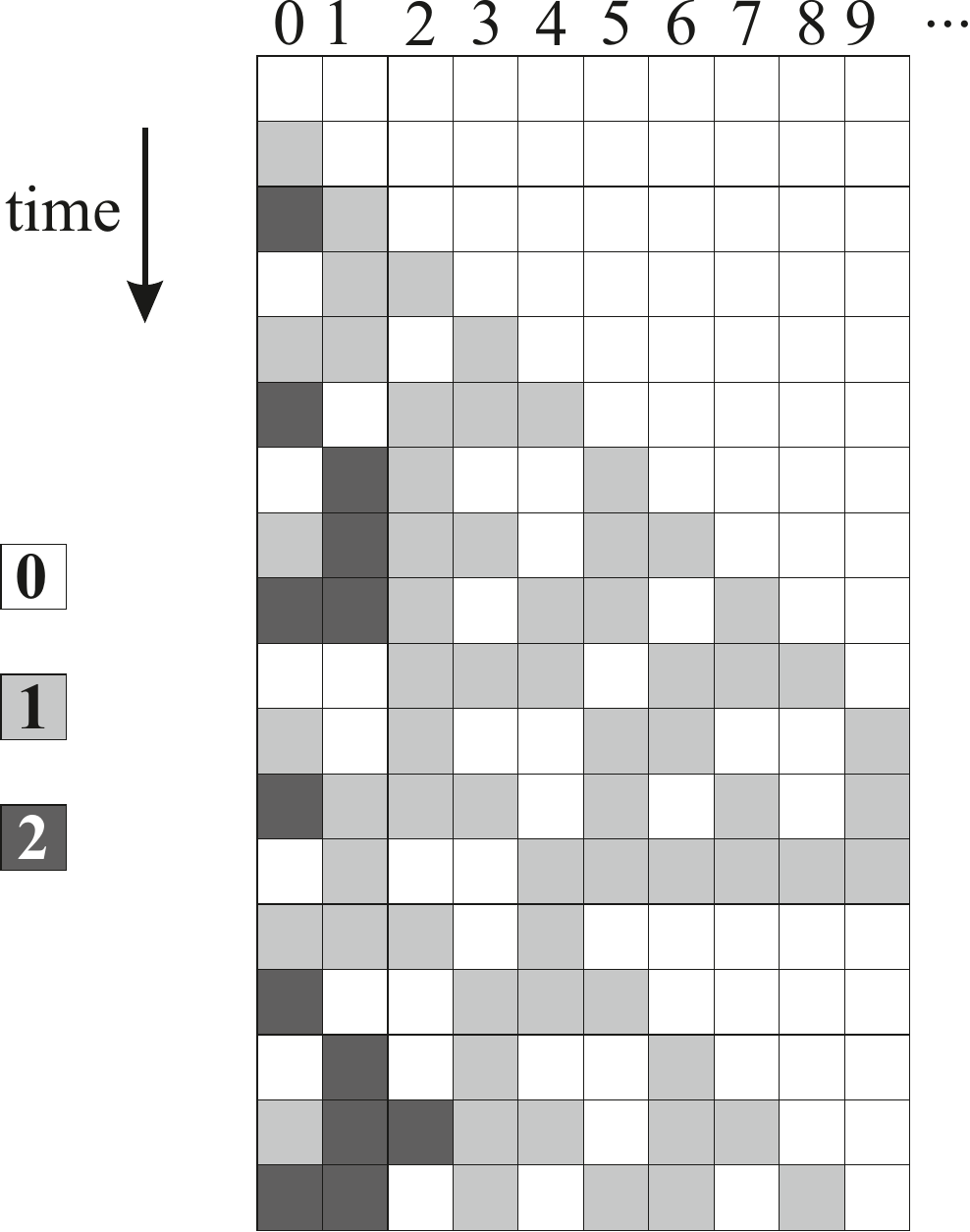"}
	\caption{The evolution of the first ten cells for 18 time steps in the three state odometer. Initially all cells are in state $0$.}
    \label{fig:odometerspacetime}
\end{center}
\end{figure}

In the following we prove that, for every $n\geq 1$, the patterns in the 
first $n$ cells cycle through all the $3^n$ words in $\{0,1,2\}^n$. This happens regardless of the states to the right, as the neighborhood of the NUCA does not involve 
any cells to the right of a cell. It follows that the three state odometer is a minimal dynamical system. From the fact that cells to the right of a cell are not among its neighbors, it directly follows that the system is equicontinuous and not sensitive. In fact, the  Auslander–Yorke dichotomy theorem states that every minimal 
system is either equicontinuous or sensitive~\cite{auslander}.


\begin{definition}
Let $o \in \Sigma^\Z$ be the uniform configuration of $0$'s (for all $x \in \N$, $o(x)=0$). For $x \in \N$, we denote $T_x:\N \rightarrow \Sigma$, where $T_x(t) = H^t(o)(x)$ for all $t \in \N$. We call $T_x$ the \emph{trace} (of $o$) at $x$. For $i,j \in \N$, $i \leq j$, we denote $T_x[i,j] = (T_x(i), \ldots, T_x(j))$, and for $a \in \Sigma$, we denote 
\begin{align*}
T_x[i,j]\#a = |\{k \in \N | i \leq k \leq j, T_x(k) = a \}|,
\end{align*}
that is, the number of occurrences of the symbol $a$ in the trace $T_x$ between time $i$ and $j$.
\end{definition}

First we show that if there are two identical sections of the trace at $x$, offset by $k$, which contain only $0$'s and $a$'s for either $a=1$ or $a=2$, and the trace at $x+1$ has 0 at the beginning of one of these sections and $a$ at the other, then if there is a symbol $a$ at one of the sections at time $t$ in the trace at $x+1$, then there is a $0$ in the other section at time $t+k$ (and vice versa). An illustration is given in Figure \ref{fig_tracecopy}.

\begin{figure}
\begin{center}
\includegraphics[scale=1.4]{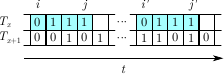}
\caption{Two identical sections of trace $T_x$ with only symbols $0$ and $1$, and a $0$ at the beginning of one section in trace $T_{x+1}$ and $1$ at the beginning of the other. This forces each cell in the sections of trace $T_{x+1}$ to have different symbols.} \label{fig_tracecopy}
\end{center}
\end{figure}

\begin{lemma} \label{odomlemma1}
Let $x\in \N$, $\{a,b\}=\{1,2\}$, and let $0\leq i\leq j$, $k > 0$, and $i' = i+k$, $j'=j+k$. Suppose $T_x[i,j]\#b=0$, $T_x[i,j] = T_x	[i',j']$ and $\{T_{x+1}(i),T_{x+1}(i')\}=\{0,a\}$. 
When $i\leq t \leq j+1$ we have $\{T_{x+1}(t),T_{x+1}(t+k)\}=\{0,a\}$. 
\end{lemma}

\begin{proof}
We use mathematical induction on $t$. The claim for $t=i$ is true by
the given assumption $\{T_{x+1}(i),T_{x+1}(i')\}=\{0,a\}$. For inductive step, let
$i\leq t \leq j$, $t' = t+k$, and suppose that $\{T_{x+1}(t),T_{x+1}(t')\}=\{0,a\}$. If $T_x(t) = T_x(t') = 0$, then  by the local rule 
\begin{align*}
T_{x+1}(t+1)=T_{x+1}(t) \mbox{ and } T_{x+1}(t'+1) = T_{x+1}(t')
\end{align*}
so that  $\{T_{x+1}(t+1),T_{x+1}(t'+1)\}=\{0,a\}$, as claimed. If $T_x(t) = T_x(t') = a$
then the states $0$ and $a$ get swapped at the cell $x+1$ so that
\begin{align*}
T_{x+1}(t+1)=T_{x+1}(t') \mbox{ and } T_{x+1}(t'+1) = T_{x+1}(t).
\end{align*}
Thus we have again  $\{T_{x+1}(t+1),T_{x+1}(t'+1)\}=\{0,a\}$.
\qed
\end{proof}

\begin{lemma} \label{odomlemma2}
Let $x\in \N$, $\{a,b\}=\{1,2\}$, and suppose for some $i,j,k \in \N$, $i\leq j$, it holds that, $T_x[i,j]\# a=2k+1$ and $T_x[i,j]\#b = 0$. If $T_{x+1}(i) = 0$ then $T_{x+1}(j+1)=a$.
\end{lemma}

\begin{proof}
For each $t\in [i,j]$, by the local rule, $T_{x+1}(t) \neq T_{x+1}(t+1)$ if and only if $T_x(t) = a$. Then because there are an odd number of $a$ symbols in the trace section $T_x[i,j]$, the symbol changes between $a$ and $0$ an odd number of times in the trace section $T_{x+1}[i,j+1]$, and hence $T_{x+1}(j+1)=a$. \qed
\end{proof}

Now we are ready to prove the main result.

\begin{theorem}
The three state odometer is minimal and not sensitive. 
\end{theorem}

\begin{proof}
Since  cells can only be affected by finitely many cells to their left, the three state odometer is equicontinuous and not sensitive. To prove minimality we show, using mathematical induction on $x$, that for any $x \in \N$ and any $w \in \Sigma^{[0,x]}$, there is $t\in \N$ such that $H^t (o)_{|[0,x]}=w$. This implies that 
 the orbit $\mathrm{orb}(o)$ of $o$ is dense. It then follows from equicontinuity that
 the orbits $\mathrm{orb}(c)$ of all configurations $c$ are dense, thus proving minimality of the system~\cite{auslander}. In fact, for all configurations $c$ the sequence of patterns 
 $(H^t (c)_{|[0,x]})_{t\geq 0}$ is a periodic sequence containing all patterns $w \in \Sigma^{[0,x]}$.

The idea of the proof is that if the period of a trace at $x$ can be divided into three parts, a leading $0$, a block of only $0$'s and $1$'s, and another block of the same length of only $0$'s and $2$'s, then this forces the trace at $x+1$ a three times longer period that can be divided in the same way. 

Suppose for some $n,k \in \Z_+$, the trace $T_x$ has period $p=2n+1=3^{x+1}$, where for all $i \in \N$, $T_x(pi)= 0$, and if we denote
\begin{align*}
&\alpha_i = pi+1, &&\alpha'_i = pi+n, &&&A_i = [\alpha_i , \alpha'_i], &&&&T_{x'}(A_i) = T_{x'}[\alpha_i , \alpha'_i]\\
&\beta_i = pi+n+1, &&\beta'_i = pi+2n, &&&B_i = [\beta_i , \beta'_i], &&&&T_{x'}(B_i) = T_ {x'}[\beta_i , \beta'_i],
\end{align*}
where $x' \in \N$, then $T_x(A_i)\#1=T_x(B_i)\#2 = 2k+1$ and $T_x(A_i)\#2=T_x(B_i)\#1 = 0$. 
That is, each period begins with a $0$, then the next $n$ symbols contain $2k+1$ occurences of $1$ and no $2$'s, and the last $n$ contain $2k+1$ occurences of $2$ and no $1$'s. This is clearly true when $x=0$ (in this case $n=1$ and $k=0$). We call $A_i$ and $B_i$ a \emph{$1$-block} and a \emph{$2$-block} respectively. This is illustrated in Figure \ref{fig_periodstructure}.

\begin{figure}
\begin{center}
\includegraphics[scale=1.4]{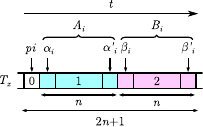}
\caption{Structure and labeling scheme for a period of the trace $T_x$.}\label{fig_periodstructure}
\end{center}
\end{figure}

We show that then the trace $T_{x+1}$ has period $3p$ where for all $i \in \N$, $T_{x+1}(3pi)=0$ and 
\begin{align*}
&T_{x+1}[\alpha_{3i},\alpha'_{3i+1}]\#1=T_{x+1}[\beta_{3i+1},\beta'_{3i+2}]\#2 = 2n+1,\\
&T_{x+1}[\alpha_{3i},\alpha'_{3i+1}]\#2=T_{x+1}[\beta_{3i+1},\beta'_{3i+2}]\#1 = 0.
\end{align*}

Suppose for some $i \in \N$, $T_{x+1}(3pi)=0$ (and hence $T_x(3pi)=0$). This is obviously true when $i=0$. Then $T_{x+1}(3pi+1) = 0$ and because $T_x(A_{3i})$ is a $1$-block, containing $2k+1$ occurences of symbol $1$, by Lemma \ref{odomlemma2}, we have $T_{x+1}(\beta_{3i}) = 1$. Let $m=T_{x+1}(A_{3i})\#1$. Now, as a $2$-block, the number of $1$'s $T_x(B_{3i}) \# 1 = 0$, meaning $T_{x+1}(\beta_{3i}+t)=1$ when $0\leq t \leq n$, giving $T_{x+1}(B_{3i})\#1 = n$. 

Now at the beginning of the next period in the trace at $x$ we have $T_{x+1}(\beta'_{3i}+1) = T_{x+1}(3pi+p) = 1$, and because $T_x(3pi+p)=0$, the local rule gives $T_{x+1}(\alpha_{3i+1})=1$ for the beginning of the next $1$-block. Then because the trace sections $T_x(A_{3i+1}) = T_x(A_{3i})$ are $1$-blocks, by Lemma \ref{odomlemma1}, $T_{x+1}(\alpha_{3i+1}+t)=0$ when $T_{x+1}(\alpha_{3i}+t)=1$ and $T_{x+1}(\alpha_{3i+1}+t)=1$ when $T_{x+1}(\alpha_{3i}+t)=0$, where $0 \leq t \leq n$. Therefore $T_{x+1}(A_{3i+1})\#1 = n-m$, and $T_{x+1}(\beta_{3i+1})=0$. 

Now the total number of $1$'s in the trace section is
\begin{align*}
T_{x+1}[\alpha_{3i},\alpha'_{3i+1}]\#1 = &m + n + 1 + (n-m) = 2n+1 
\end{align*}
and $T_{x+1}[\alpha_{3i},\alpha'_{3i+1}]\#2 = 0$. Because $T_{x+1}(\beta_{3i+1})=0$, with a simple transformation of this argument, we can also show that $T_{x+1}[\beta_{3i+1},\beta'_{3i+2}]\#2=2n+1$, $T_{x+1}[\beta_{3i+1},\beta'_{3i+2}]\#1=0$ and $T_{x+1}(3pi + 3p)= 0$. Therefore the statement holds. 

\begin{figure}
\begin{center}
\includegraphics[scale=1.4]{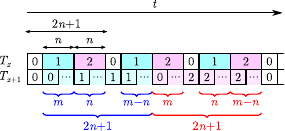}
\caption{Illustration of the argument for why the period in trace $T_x$ forces the trace $T_{x+1}$ to have a similar but longer period. In the trace $T_x$ the coloured blocks labeled with $1$ and $2$ are $1$-blocks and $2$-blocks respectively. The numbers below in blue and red (left and right) are the number of $1$ and $2$ symbols in the denoted section of trace $T_{x+1}$  respectively.} \label{fig_odometerind}
\end{center}
\end{figure}

Then because $T_0$ has a period of this kind of length $3$, every trace $T_x$ has such a period of length $3^{x+1}$. Now, suppose for some $x\in \N$, for every $u \in \Sigma^{[0,x]}$ there is some $t\in \N$ such that $H^t (o)_{|[0,x]}=u$. This is clearly true when $x=0$. Let $w \in \Sigma^{[x+1]}$, and let $u \in \Sigma^{[0,x]}$ be such that for some $a \in \Sigma$, the word $w = ua$. Let $p,n\in \N$ be such that $p=2n+1$ is the period of trace $T_x$. Now by the previous, for some $0 \leq t < p$, $H^{t+ip} (c)_{|[0,x]}=u$ for all $i \in \N$. Now there are three cases depending on where in the period of $t$ is.

\emph{Case 1:} $t=0$. Then by the previous there is $i \in \N$ such that $H^{pi}(x+1) = 0$, $H^{pi+p}(x+1) = 1$ and $H^{pi+2p}(x+1) = 2$. Therefore the statement holds.

\emph{Case 2:} $1 \leq t \leq n$. Then for some $i\in \N$, $H^{t+pi} (x+1) \in \{0,1\}$. Then because $t$ is part of a $1$-block in the trace at $x$, by the previous $H^{t+pi+p} (x) \in \{0,1\}$ and $H^{t+pi+p} (x+1) \neq H^{t+pi} (x+1)$ and $H^{t+pi+2p} (x+1) = 2$. Hence the statement holds. 

\emph{Case 3:} $n+1 \leq t \leq 2n$. Then for some $i\in \N$, $H^{t+pi} (x+1) \in \{0,2\}$. Then because $t$ is part of a $2$-block in the trace at $x$, by the previous $H^{t+pi+p} (x) \in \{0,2\}$ and $H^{t+pi+p} (x+1) \neq H^{t+pi} (x+1)$ and $H^{t+pi+2p} (x+1) = 1$. Hence the statement holds.

Therefore, the NUCA $H$ is minimal. Since every cell's neighbourhood only includes the cell itself and (at most) a cell to its left, and since any cell only has finitely many cells to it's left, the NUCA is obviously not sensitive. 
\qed 
\end{proof}

By the preceding, the three state odometer $H$ clearly has no periodic points. If we extend its rule distribution to $\Z$ such that the cells in $\Z \setminus \N$ use the right shift rule $f_\rightarrow$ of Example~\ref{ex:firstexample}, we obtain a NUCA on $\Z$ which is transitive, but has no periodic points. This is in contrast to some previous examples of NUCA with no periodic points, which have been non-transitive \cite{KamilyaK21}.

\section{A minimal, non-sensitive two-dimensional NUCA}

Embedding the three state odometer along a snake that covers $\Z^2$ we obtain a  two-dimensional NUCA that is minimal, equicontinuous and non-sensitive.

\begin{theorem}
There exists a $2$-dimensional NUCA that is minimal -- and hence also transitive -- but not sensitive.
\end{theorem}

\begin{proof}
There clearly exists a bijection $s:\N\longrightarrow \Z^2$ with the following properties:
\begin{itemize}
\item $s(0)=(0,0)$,
\item for every $k>0$ we have $s(k)-s(k-1)\in\{(\pm1,0),(0,\pm 1)\}$.
\end{itemize}
An example of such an enumeration of $\Z^2$ is the spiral shown in Figure~\ref{fig:spiral}(a). 

\begin{figure}
\begin{center}
\includegraphics[scale=0.3]{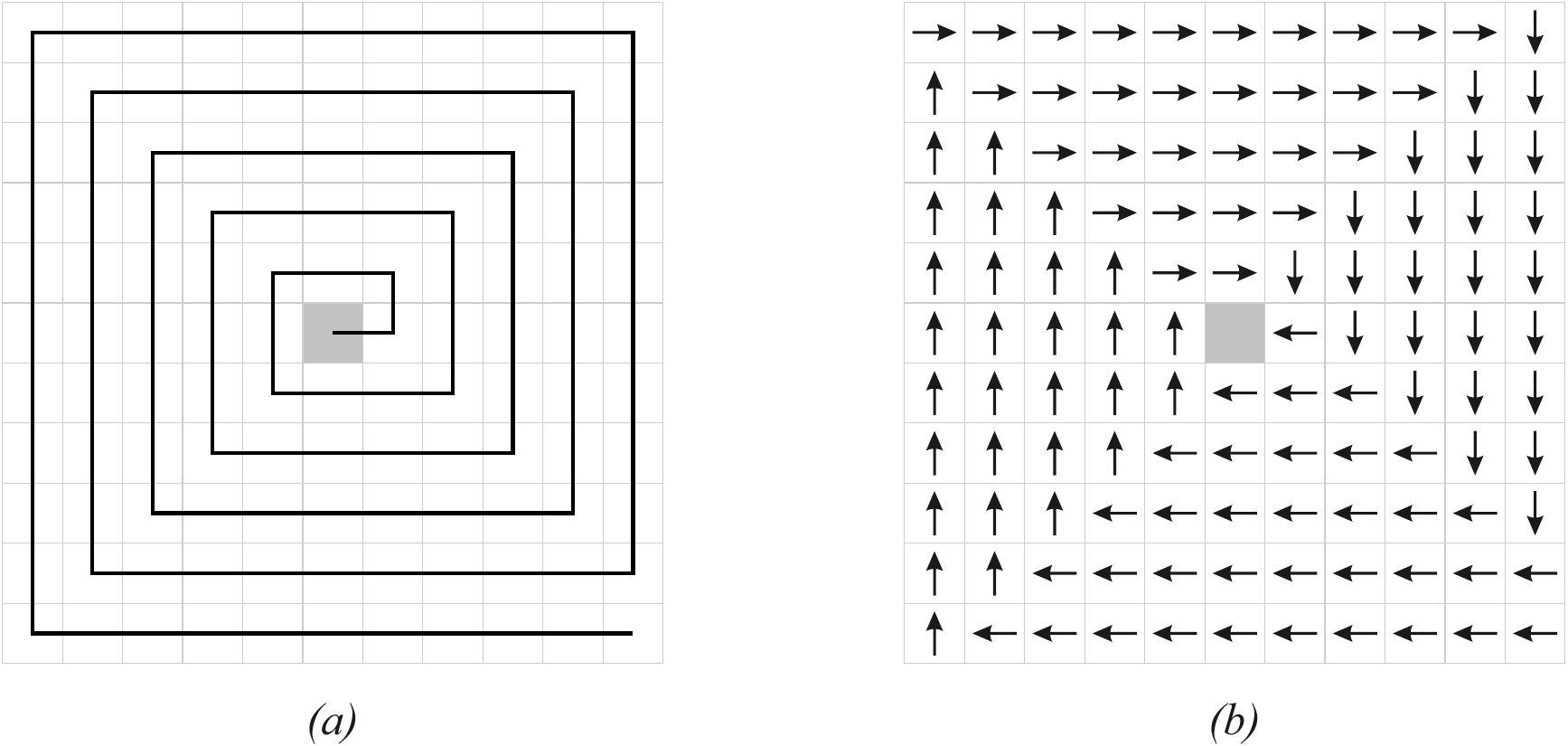}
\caption{(a) A spiral ordering $s:\N\longrightarrow \Z^2$. The cell marked grey is the origin $s(0)=(0,0)$. For every $k$, the $k$'th cell along the spiral is the cell $s(k)$. (b) The rule distribution $\theta$ of the NUCA: The grey cell uses the local rule $g$, while cells with an arrow use $f$ with the neighbor in the direction of the arrow.}\label{fig:spiral}
\end{center}
\end{figure}

Our NUCA uses five local rules: rule $g$ of Table~\ref{table:odometer1_rule1}
at cell $(0,0)$, and for every $k>0$ the rule $f$ of Table~\ref{table:odometer1_rule2}
at the cell $s(k)$ with its neighbor $s(k-1)$. More precisely, in Table~\ref{table:odometer1_rule2} cell $i$ refers to $s(k)$ and cell $(i-1)$ to $s(k-1)$. Thus there are four local rules that use $f$, in four different orientations. Figure~\ref{fig:spiral}(b) shows the rule distribution $\theta$.

Because the cell $s(k)$ is only affected by the finitely many 
cells $s(0),\dots , s(k)$, the NUCA we constructed is equicontinuous and non-sensitive. 
As the three state odometer has the property that, for 
any positive integer $n$, the pattern in the first $n$ cells cycle through all
words of length $n$, regardless of the initial configuration, the two-dimensional NUCA we constructed have this same property for cells $s(0),\dots , s(n-1)$. It follows that
the NUCA is minimal. \qed
\end{proof}

\section{Recurrent rule distributions}

We have seen that there are transitive NUCAs that are not necessarily sensitive. However, if the rule distribution is recurrent then transitivity implies sensitivity.

\begin{theorem}
Let $\theta \in \mathcal{R}^{\Z^d}$ be a recurrent rule distribution such that $H_\theta: \Sigma^{\Z^d} \rightarrow \Sigma^{\Z^d}$ is transitive. Then $H_\theta$ is sensitive.
\end{theorem}

\begin{proof}
Suppose $H_\theta$ is not sensitive. Let us first prove that, for some $n>0$,  the traces of all configurations at cell $\vec{0}$ have period $n$.

Because $H_\theta$ is not sensitive, there is some $D \subseteq \Z^d$ and $c \in \Sigma^{\Z^d}$ such that for any $e \in \cyl (c,D)$ and $t \geq 0$, $H_\theta^t(e)(\vec{0}) = H_\theta^t(c)(\vec{0})$. Because $H_\theta$ is transitive, there is some $e \in \cyl(c,D)$ and $n>0$ such that $H_\theta^n (e) \in \cyl(c,D)$. Then for any $t \geq 0$, 
\begin{align*}
    H_\theta^t(c)(\vec{0}) = H_\theta^t(H_\theta^n(e))(\vec{0}) = H_\theta^{n+t}(e)(\vec{0}) = H_\theta^{n+t}(c)(\vec{0}),
\end{align*}
meaning that the trace $T = (H_\theta^t (c)(\vec{0}))_{t\geq 0}$ of $c$ at cell $\vec{0}$ has period $n$.

Consider then an arbitrary configuration $e\in \Sigma^{\Z^d}$, and denote  
$a_t = H_\theta^t (e)(\vec{0})$ for all $t\geq 0$.
Let $r>0$ be such that the rules in $\R$ have at most radius $r$. Obviously for any $e' \in \cyl (e,[-nr,nr]^d)$, we have $H_\theta^{t} (e')(\vec{0}) = a_{t}$ when $0 \leq t\leq n$. Because $H_\theta$ is transitive, there is some $k\geq 0$ such that 
$H_\theta^k (\cyl(c,D)) \cap \cyl(e,[-nr,nr]^d) \neq \emptyset$. Then $a_0, \ldots, a_n$ is part of the trace $T$ starting at time $k$, that is, $a_{t}=T_{k+t}$ when $0 \leq t \leq n$. Because $T$ has period $n$, we have $a_0 = a_n$. As $e$ was arbitrary, it follows that the trace $(H_\theta^t (e)(\vec{0}))_{t\geq 0}$ of every configuration $e$ at cell $\vec{0}$ has period $n$.

\begin{figure}
\begin{center}
	\includegraphics[scale=1]{"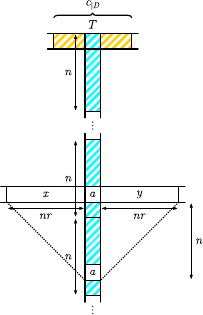"}
	\caption{Illustration, in dimension $d=1$, of the argument for why any cell at the centre of a copy of the rule patters $\theta_{|[-nr,nr]^d}$ must repeat in the trace after $n$ steps. Any configuration that agrees with $c$ within $D$ will have a trace of period $n$ at cell $\vec{0}$. Then for any pattern $xay$ extending symbol $a$ by $nr$ cells in all directions, after some time some configuration in $\cyl (c,D)$ will have $xay$ at the origin. Any configuration with a copy of $xay$ on a copy of the rule pattern $\theta_{|[-nr,nr]^d}$ will have the same trace beginning from the $a$-cell for the next $n$ steps. Because that section of the trace must be a part of the period $n$ trace at the origin, the symbol $a$ must recur after $n$ steps.}
\end{center}
\end{figure}

Because $\theta$ is recurrent, there is $\vec{m}\neq\vec{0}$ such that 
$\sigma_{\vec{m}}(\theta)_{|[-nr,+nr]^d} = \theta_{|[-nr,+nr]^d}$. 
Clearly when $0\leq t \leq n$, for any $e \in \Sigma^{\Z^d}$, it holds that  $H_\theta^{t} (e)(\vec{m}) = H_\theta^{t} (\sigma_{\vec{m}}(e))(\vec{0})$. In particular, for all $e \in \Sigma^{\Z^d}$ and all $t\in\N$ 
\begin{align*}H_\theta^{t+n} (e)(\vec{m}) &=
H_\theta^n(H_\theta^{t} (e))(\vec{m}) \\&=
H_\theta^{n} (\sigma_{\vec{m}}(H_\theta^{t} (e)))(\vec{0}) \\&=
H_\theta^{0} (\sigma_{\vec{m}}(H_\theta^{t} (e)))(\vec{0}) \\&=
H_\theta^{0} (H_\theta^{t} (e))(\vec{m}) \\&=
H_\theta^{t} (e)(\vec{m})
\end{align*}
implying that also at cell $\vec{m}$ the trace 
$(H_\theta^t (e)(\vec{m}))_{t\geq 0}$ of every configuration $e$ has period $n$.

But this contradicts transitivity: let $c$ be a uniform configuration with the same state at 
all cells. 
Let $D'=[-nr,nr]^d\cup (\vec{m}+[-nr,+nr]^d)$, and consider any $e \in \cyl (c,D')$. Then 
$\sigma_{\vec{m}}(e)_{|[-nr,nr]^d} =e_{|[-nr,nr]^d}$ and so 
$H_\theta^{t} (e)(\vec{m})=H_\theta^{t} (e)(\vec{0})$ when $0\leq t \leq n$. Because the traces
of $e$ at cells $\vec{m}$ and $\vec{0}$ both have period $n$, we have that 
$H_\theta^{t} (e)(\vec{m})=H_\theta^{t} (e)(\vec{0})$ for all $t\geq 0$. Because 
$e \in \cyl (c,D')$ is arbitrary, this means that for all $t\geq 0$ the intersection of
$H_\theta^t(\cyl (c,D'))$ and $\cyl (c',\{\vec{0},\vec{m}\})$ is empty, where $c'$ is a configuration such that $c'(\vec{0})\neq c'(\vec{m})$.

\qed
\end{proof}

\section{Conclusions}

In uniform cellular automata, transitivity implies sensitivity. Additionally, transitivity, syndetic transitivity, total transitivity and weak mixing are mutually equivalent, and both mixing and strong transitivity imply transitivity. However, it is open whether transitivity or strong transitivity implies mixing. 
A simple counter-example (Example~\ref{ex:firstexample}) shows that there are strongly transitive (and hence transitive) non-uniform cellular automata, which are not mixing, weakly mixing, or totally transitive. 

We've shown that if a NUCA is defined by a recurrent rule distribution, it is sensitive if it is transitive. We've also shown an example of a minimal (and hence transitive), non-sensitive NUCA on $\N$ (as a subset of $\Z$), and a minimal, non-sensitive two-dimensional NUCA, which can easily be extended to higher dimensions. It remains open whether there exists a transitive and non-sensitive NUCA on $\Z$. It is also open whether all transitive NUCA are syndetically transitive.

We finish by proposing another candidate for a three state odometer on $\N$, similar to the one described in Section~\ref{sec:odometer} but with a different local rule in cells $i>0$. Let $\Sigma = \{0,1,2 \}$ and let $h$ be the local rule with neighborhood $\{ -1,0\}$ that uses the transition function given in Table \ref{table:odometer2_rule2}. Let $g$ be, as in Section~\ref{sec:odometer}, the three periodic rule in Table~\ref{table:odometer1_rule2}. 
 Let $\phi \in \{h,g\}^\N$ be the local rule distribution $g\ h\ h\ h \ldots$ on $\N$. The NUCA $H_\phi$ appears to exhibit similar behavior to the three state odometer described previously, though we have not proved this. Notably, because $h$ is identical to $g$ when the left neighbor is in the fixed state $0$, the NUCA $H_\phi$ can also be defined using the uniform rule distribution $h\ h\ h \ldots$ on $\N$ with a constant boundary condition, where the left neighbor of cell $0$ is fixed to state $0$.

\begin{table}[!h] 
    \centering
    \begin{tabular}{c|c|c|c}
         \textbf{Cell $(i-1)$} & \textbf{Cell $i$} & \textbf{Next state of cell $i$} & 
 \\ \hline
        0 & 0 & 1 & \\
        0 & 1 & 2 & 
        \fbox{\begin{tabular}{c}
           $0 \rightarrow 1 \rightarrow 2 \rightarrow 0$
        \end{tabular}}
        \\
        0 & 2 & 0 & \\ 
        \hline
        
        1 & 0 & 1 & \\
        1 & 1 & 0 & 
        \fbox{\begin{tabular}{c}
           $0 \leftrightarrow 1$
        \end{tabular}}
        \\
        1 & 2 & 2 & \\
        \hline
        
        2 & 0 & 2 & \\
        2 & 1 & 1 &
        \fbox{\begin{tabular}{c}
           $0 \leftrightarrow 2$
        \end{tabular}}
        \\
        2 & 2 & 0 & \\
        \hline
    \end{tabular}
    \caption{The local rule at cells $i>0$ in the second candidate for a three state odometer.}
    \label{table:odometer2_rule2}
\end{table}

\bibliographystyle{splncs04}
\bibliography{biblio}

@article{KamilyaK21,
  author       = {Supreeti Kamilya and
                  Jarkko Kari},
  title        = {Nilpotency and periodic points in non-uniform cellular automata},
  journal      = {Acta Informatica},
  volume       = {58},
  number       = {4},
  pages        = {319--333},
  year         = {2021}
}

@article{PaturiNACO,
  author       = {Katariina Paturi and Jarkko Kari},
  title        = {On the surjunctivity and the {G}arden of {E}den theorem for non-uniform cellular automata},
  journal      = {Natural Computing},
  volume       = {25},
  number       = {9},
  year         = {2026}
}

@InProceedings{PaturiUCNC,
  author =	{Katariina Paturi},
  title =	{{Reversibility, balance and expansitivity of non-uniform cellular automata}},
  booktitle =	{22nd International Conference on Unconventional Computation and Natural Computation (UCNC 2025)},
  pages =	{17--32},
  series =	{Lecture Notes in Computer Science (LNCS)},
  year =	{2026},
  volume =	{16364}
}

@article{CodenottiMargara96,
author = {Bruno Codenotti and Luciano Margara},
title = {Transitive Cellular Automata Are Sensitive},
journal = {The American Mathematical Monthly},
volume = {103},
number = {1},
pages = {58--62},
year = {1996}
}

@article{Kurka1997, 
title={Languages, equicontinuity and attractors in cellular automata}, 
volume={17}, 
number={2},
journal={Ergodic Theory and Dynamical Systems},
author={K\r{u}rka, Petr}, 
pages={417–433},
year={1997}
}

@article{Myhill1963,
  title={{The converse of Moore's Garden-of-Eden theorem}},
  author={Myhill, John},
  journal={Proceedings of the American Mathematical Society},
  volume={14},
  number={4},
  pages={685--686},
  year={1963}
}

@article{Moore1962,
  title={Machine models of self-reproduction},
  author={Moore, Edward F},
  journal={Proceedings of Symposia in Applied Mathematics},
  volume={14},
  pages={17--33},
  year={1962}
}

@article{auslander,
author = {Joseph Auslander and James A. Yorke},
title = {{Interval maps, factors of maps, and chaos}},
volume = {32},
journal = {Tohoku Mathematical Journal},
number = {2},
pages = {177 -- 188},
year = {1980}
}

@article{Moothathu2005,
title = {Homogeneity of surjective cellular automata},
journal = {Discrete and Continuous Dynamical Systems},
volume = {13},
number = {1},
pages = {195-202},
year = {2005},
author = {T.K. Subrahmonian Moothathu},
}
\end{document}